\documentclass[11pt,twoside]{article}
\usepackage{latexsym}
\usepackage{amssymb,amsbsy,amsmath,amsfonts,amssymb,amscd,amsthm}
\usepackage{graphicx,color}
\newcommand{\fer}[1]{(\ref{#1})}

\newcommand{\commentout}[1]{}
\newcommand{\R}{\mathbb{R}}

\newcommand{\Z}{\mathbb{Z}}

\newcommand {\al} {\alpha}
\newcommand {\bt} {\beta}
\newcommand {\e}  {\varepsilon}
\newcommand {\da} {\delta}

\newcommand {\vp} {\varphi}

\newcommand {\Chi} {{\bf \raise 2pt \hbox{$\chi$}} }
\newcommand {\caa} { {\mathcal A} }
\newcommand {\cbb} { {\mathcal B} }
\newcommand {\cac} { {\mathcal C} }
\newcommand {\cad} { {\mathcal D} }

\newcommand {\f}   {\frac}
\newcommand {\p}   {\partial}
\newcommand{\beq}{\begin{equation}}
\newcommand{\eeq}{\end{equation}}
\newcommand{\bea} {\begin{array}{rl}}
\newcommand{\eea} {\end{array}}
\newcommand{\bc} {\begin{cases}}
\newcommand{\ec} {\end{cases}}
\newcommand{\bepa}{\left\{ \begin{array}{l}}
\newcommand{\eepa} {\end{array}\right.}
\newtheorem{theorem}{Theorem}[section]
\newtheorem{lemma}[theorem]{Lemma}

\newtheorem{proposition}[theorem]{Proposition}

\title{\Large \bf A singular Hamilton-Jacobi equation modeling the tail problem}

\author{Sepideh Mirrahimi  \thanks{
UPMC, CNRS UMR 7598, Laboratoire Jacques-Louis Lions, F-75005, Paris.
Email: mirrahimi@ann.jussieu.fr}
\and Guy Barles  \thanks{
Laboratoire de Math\'ematiques et Physique Th\'eorique, CNRS UMR 6083, F\'ed\'eration Denis Poisson,
Universit\'e Fran\c{c}ois Rabelais, Parc de Grandmont,
37200 Tours, France.
Email: barles@lmpt.univ-tours.fr}
\and  Beno\^ \i t Perthame \footnotemark[1] \thanks{
INRIA EPI BANG and Institut Universitaire de France. Email: benoit.perthame@upmc.fr}
\and Panagiotis E. Souganidis  \thanks{The University of Chicago,
Department of Mathematics, 5734 S. University Avenue,
Chicago, IL 60637, USA. Email: souganidis@math.uchicago.edu}
\thanks{Partially supported by the National Science Foundation
}}

\date{\today}

\begin{document}
\maketitle
\pagestyle{plain}
\pagenumbering{arabic}

\begin{abstract}
We study the long-time/long-range behavior of reaction diffusion equations with negative square root-type reaction terms.
In particular we investigate the exponential behavior of the solutions after space and time are scaled in a hyperbolic way by a small parameter. This leads to a new type of quasi-variational inequality for a Hamilton-Jacobi equation. The novelty is that the obstacle, which
defines the open set where the solutions of the reaction diffusion equation do not vanish in the limit, depends on the solution itself. Counter-examples show a nontrivial lack of uniqueness for the variational  inequality depending on the  conditions imposed on the free boundary of  this open set. Both Dirichlet and state constraints boundary conditions play a role. When the competition term does not change sign,  we can identify the limit while, in general, we only obtain lower and upper bounds.

Although models of this type are rather old and extinction phenomena are as important as blow-up, our motivation comes from the so-called  ``tail problem'' in population biology. One way to avoid meaningless exponential tails is to impose a singular mortality rate below a given survival threshold. Our study shows that the precise form of this singular mortality term is asymptotically irrelevant and that, in the survival zone, the population profile is impacted by the survival threshold (except in the very particular case when the competition term is nonpositive).

\end{abstract}

\\
{\bf AMS Class. No}: 35B25, 35K57, 49L25, 92D15

\section{Introduction}
\label{sec:int}

We study the asymptotic behavior, as $\e\to0$, of the solutions to reaction-diffusion equations (with singular reaction term) of the form
\begin{equation}
\label{eq:sroot}
\begin{cases}
n_{\e,t}-\e \Delta n_\e = \frac{1}{\e} n_\e R-\frac{1}{\e}{(\beta_\e n_\e)}^{1/2} \quad \hbox{in  }\R^d \times (0,+\infty),
\\
n_\e = e^{u_\e^0/\e} \quad \text {on  } \quad \R^d \times \{0\},
\end{cases}
\end{equation}
where $u_\e^0: \R^d\to \R$ is a given function, $R:\R^d\to \R$  represents a linear  logistic growth/death rate and the survival threshold parameter $\beta$, which models a singular death term, is given, for some $u_m<0$, by
\begin{equation}
\label{eq:threshold}
\beta_\e= e^{u_m/\e}.
\end{equation}

The positive parameter $\e$ is introduced by a hyperbolic scaling $(x,t) \mapsto (x/\e,t/\e)$ with the aim to describe the long time and long range behavior of the unscaled problem (corresponding to $\e =1$). The limiting behavior of scaled reaction-diffusion equations with KPP-type reaction has been studied extensively in, among other places, the theory of front propagation (\cite{GB.LE.PS:90,PS:98,LE.PS:89}) using the so called WKB-(exponential) change of the unknown.

The novelty of the problem we are considering here is the presence of the negative square root term. To the best of our knowledge, the first study of such nonlinearity  goes back to \cite{LE.BK:79} where it is proved that local extinction occurs, i.e.,  the solution can vanish in a domain and stay positive in another region. For this reason $\beta$ is thought to represent a survival threshold. That a solution of a parabolic problem can vanish locally is a surprising effect and as singular as the blow-up phenomena for supercritical reactions terms (\cite{PQ.PS:07}). In population biology such behavior prevents the so-called ``tail problem''  where very small (and thus meaningless) populations can generate artifacts (\cite{MG.BP:09}).
Although the mathematical analysis of the limit of \fer{eq:sroot} turns out to be a full subject in itself, our primary motivation comes from qualitative questions in population dynamics.

Indeed \eqref{eq:sroot} is the simplest model for studying the effect of ``cutting the tail'' but many other problems are relevant in ecology. Along the same lines, in the context of front propagation, one may consider the modified Fisher--KPP equation
$$
n_{\e,t}-\e \Delta n_\e = \frac{1}{\e} n_\e (1- n_\e) -\frac{1}{\e}{(\beta_\e n_\e)}^{1/2}\quad \hbox{in  }\R^d \times (0,+\infty),
$$
and ask the question whether the square root term changes fundamentally the study in \cite{LE.PS:89} and \cite{MG.BP:09} of the propagation of the invading/combustion fronts. In the context of speciation, an  elementary model in adaptive evolution is the non-local reaction-diffusion equation
$$
n_{\e,t}-\e \Delta n_\e = \frac{1}{\e} n_\e R(x,I_\e)-\frac{1}{\e}{(\beta_\e n_\e)}^{1/2}\quad \hbox{in  }\R^d \times (0,+\infty) \quad \text{ with } \quad I_\e (t) =\int \psi(x)n_\e(x,t)dx,
$$
where $n_\e$ is the population density of individuals with phenotypical trait $x$, $R$ represents the net growth rate, $\psi$ is the consumption rate of individuals and $I(t)$ is the total consumption of the resource at time $t$. The survival threshold was introduced in \cite{MG.BP:09}. Finally $\e$ may represent large time and small mutations as studied in \cite{GB.SM.BP:09,GB.BP:07,GB.BP:08}. It is known that under some assumptions the density concentrates as an evolving Dirac mass for the fittest trait. In biological terms this means that one or several dominant traits survive while others become extinct. Phenomena such as the discontinuous jumps of the fittest trait, non smooth branching and fast dynamics compared to stochastic simulations,  motivated \cite{MG.BP:09} to improve the model by including a survival threshold. Numerical results confirm that this modification  gives dynamics comparable to stochastic models. It is interesting to investigate rigorously whether the dynamics of the Dirac concentration points are really changed by the survival threshold and to explain why its specific form ($n_\e^{1/2}$ versus $n_\e^\gamma $ with $0< \gamma <1$) seems irrelevant.

A way to approach these questions for  \eqref{eq:sroot} is through the asymptotic analysis of $n_\e$. Since, as in the classical case, i.e., the Fisher-KPP equation without the square root term (see \cite{LE.PS:89}),  $n_\e$ decays exponentially, the limit is better described using the Hopf-Cole transformation
\beq\label{defue}
u_\e= \e \ln n_\e,
\eeq
which, for $u_\e^0 = \e \ln n_\e^0$, leads to the ``viscous'' Hamilton-Jacobi initial value problem
\beq
\label{eq:ue}\begin{cases}
\displaystyle u_{\e,t}-\e \Delta u_\e-|D u_\e|^2=R-\exp \left(\frac{u_m-u_\e}{2 \e}\right)
\quad \text{ in  } \quad \R^d \times (0,+\infty),
\\[2mm]
u_\e=u_\e^0 \quad \text {in  } \R^d \times \{0\}.
\end{cases}
\eeq

Throughout the paper we assume that there exist $C>0$ and $u^0 \in \mathrm{C}^{0,1}(\R^d)$ such that
\begin{equation}\label{assum:R}
\| R \|_{\mathrm{C}^{0,1}}\leq C \quad \text{ and } \quad \|u^0\|_{\mathrm{C}^{0,1}} \leq C,
\end{equation}
\begin{equation}\label{test}
u^0_\e \in \mathrm{C}(\R^d), \quad u^0_\e \leq C \quad \text{ and } \quad  u^0_\e \underset{\e\to 0}{\longrightarrow}  u^0 \quad \text{ in } \quad \mathrm{C}(\R^d).
\end{equation}

In the limit $\e\to 0$, it is easy to see, at least formally, that any local uniform limit of the family $(u_\e)_{\e>0}$ will  satisfy, in the sense of the Crandall-Lions viscosity solutions (\cite{C.L:83}), the Hamilton-Jacobi free boundary problem
\beq \label{eq:hjs}
\begin{cases}
u_t =|D u|^2 + R \quad  \text{in } \quad \Omega\subset \R^d \times (0,\infty),\\
u=-\infty \quad \text{in } \quad \overline{\Omega}^c\cap (\R^d \times (0,\infty)),\\
u\geq  u_m  \quad \text{in } \quad \overline{\Omega},  \\
u=u^0 \quad \text{in } \quad \overline{\Omega}\cap (\R^d \times \{0\}),
\end{cases}
\eeq
with the space-time open set $\Omega$ defined by
$$
\Omega = {\mathcal Int} \;  \big\{ (x,t)\in \R^d \times (0,\infty) : \ \lim_{\e \to 0} u_\e (x,t) > - \infty \big\}.
$$

Notice that \fer{eq:hjs} is an obstacle problem with an obstacle depending on the solution itself. As a matter of fact the open set $\Omega$ plays an important role and, hence, the problem may be better stated in terms of the pair $(u,\Omega)$.
The difficulty is that \fer{eq:hjs} has several viscosity solutions (see Appendix \ref{ap.nuniq} for examples) depending on the sense the boundary conditions are achieved and the sign of $R$.

Next we discuss the two boundary conditions arising in \fer{eq:hjs}. The first is the Dirichlet boundary condition in the third relation in \eqref{eq:hjs}. Its precise form is
\beq\label{limbor}
\lim_{(x,t)\to (x_0,t_0) \in \partial \Omega} u(x,t) = u_m.
\eeq
The second is the state constraint boundary condition (see \cite{HS:86}), which is natural in view of the second equality in \eqref{eq:hjs}. It states that
\beq\label{sc}
u \; \text {is a  supersolution in} \; \overline{\Omega} \; \text { and a subsolution in  $\Omega$ }.
\eeq

The basic questions we are considering in this paper are:
\\
$\bullet$ What boundary condition should be satisfied by the limits of the family $(u_\e)_{\e>0}$ on $\partial \Omega$? Dirichlet or state constraint?
The latter appears to  play a fundamental role.
To the best of our knowledge, there are no results available for state constraint problems with time varying and non smooth domains. Most of the technicalities in the paper stem from this difficulty.
\\
$\bullet$ Does the limit $\e\to0$ select a particular solution to \fer{eq:hjs}, i.e., is there a natural selection? Is the limit of the family $(u_\e)_{\e>0}$ the maximal subsolution or minimal solution to \fer{eq:hjs}?
\\
$\bullet$ Do the limits of the family $(u_\e)_{\e>0}$ depend on the specific form of the survival threshold, i.e.,  can we replace $(\beta_\e n_\e)^{1/2}$ by $(\beta_\e n_\e)^{\gamma}$ with $\gamma \in (0,1)$ without affecting the outcome?
\\

An important ingredient of our analysis is the asymptotics, as $\e\to0$, of the solution $u_\e^1$ of
\begin{equation}
\label{hj1}
\begin{cases}
\displaystyle u^1_{\e,t}=\e \Delta u^1_\e + |D u^1_\e|^2 + R
\quad \text{ in  } \quad \R^d \times (0,+\infty),
\\[2mm]
u^1_\e=u_\e^0 \quad \text {in  } \R^d \times \{0\},
\end{cases}
\end{equation}
which is obtained, after the Hopf-Cole transformation
\beq
\label{defue1}
u_\e^1=\e \ln n_\e^1,
\eeq
from the simplified reaction diffusion equation
 \begin{equation}
\label{eq:simple}
\begin{cases}
n_{\e^,t}^1-\e \Delta n_\e^1 = \f{n_\e^1}{\e}R \quad \text {in  } \quad \R^d \times (0,+\infty), \\[4mm]
n_\e^1 = \exp(\e^{-1}u_\e^0) \quad \text{in  } \R^d\times\{0\}.
\end{cases}
\end{equation}

In view of \fer{assum:R} and \fer{test}, it follows from \cite{LE.PS:89}
that, as $\e \to0$, the sequence $(u_\e^1)_{\e}$ 
converge locally uniformly to $u^1\in \mathrm{C}(\R^d\times(0,\infty))$, which is the unique viscosity solution of the eikonal -type equation
\beq\begin{cases}\label{eq:u1}
u^1_t=|D u^1|^2+R\quad \text {in  } \ \R^d \times (0,+\infty),\\[4mm]
u^1=u^0 \quad \hbox{in  } \ \R^d\times\{0\}.
\end{cases}
\eeq

The maximum principle yields  $n_\e \leq n_\e^1$, which in turn implies that $u_\e \leq u_\e^1$ and, in the limit (this is made precise later), $u\leq u^1$. It also follows from \fer{eq:ue}, at least formally, that, as $\e\to0$,
$$
u_\e \to -\infty  \qquad \text{in }\; (\R^d\times (0,\infty))\backslash \overline{\Omega^1},
$$
where
\beq\label{omeg}
\Omega^1=\{(x,t)\,|\; u^1(x,t)>u_m\}.
\eeq

It turns out that the case of nonpositive rate $R$ is particularly illuminating and the above  questions can be answered completely and positively using $u^1$ (see Section \ref{sec:Rneg}).
The problem is, however, considerably more complicated  when $R$ takes positive values. In this case we introduce an iterative procedure that builds sequences of sub and supersolutions (Section \ref{Rgeneral}). This construction gives the complete limit of $u_\e$  when $R$ is constant (Section \ref{sec:Rcons}). The limit is not the maximal subsolution of~\fer{eq:hjs} and the Dirichlet condition
is not enough to select it. In Section \ref{sec:rpositive}, we consider strictly positive spatially dependent $R$ and provide a complete answer in terms of the iterative procedure. The relative roles of the Dirichlet and state constraint boundary conditions
appear clearly in this case. In Section \ref{sec:conclusion} we summarize our results. In the three part Appendix we present some examples of nonuniqueness as well as the proofs of few technical facts
used earlier.

We conclude the introduction with the definition and the notation of the half-relaxed limits that we will be using throughout the paper. To this end, if $(w_\e)_{\e>0}$ is a family of bounded functions, the upper and lower limits, which are denoted by $\bar w$ and $\underline w$ respectively, are given by
\beq\label{limits}
\overline w(x)=\underset{\e\rightarrow 0, y \rightarrow x}\limsup w_\e(y) \quad \text{ and } \quad
\underline w(x)=\underset{\e\rightarrow 0, y \rightarrow x}\liminf w_\e(y).
\eeq

\medskip
\noindent
\textbf{Acknowledgements. }
The authors wish to thank the anonymous referee for his very careful reading of the first version of this article and his numerous suggestions to improve its readibility.

\section{Nonpositive growth rate}
\label{sec:Rneg}

Here we assume
\beq\label{nonpositive}
R \leq 0 \quad \text{ in } \quad \R^d,
\eeq
and show that the behavior of the family $(u_\e)_{\e}$, in the limit $\e\to0$, can be described completely in terms of solution $u^1$ of \fer{eq:u1}, which carries all the necessary information. More precisely, we can state the

\begin{theorem} \label{thm:principal}
Assume (\ref{assum:R}), (\ref{test}) and \fer{nonpositive}. As $\e \to 0$, the family $(u_\e)_{\e>0}$ converges, locally uniformly in $\Omega^1$ and in $\left(\R^d\times (0,\infty)\right)\backslash\overline{\Omega^1}$, to
\begin{equation}\label{Rneglim}
u(x,t)=
\begin{cases}
u^1(x,t) \quad  \text{for } \quad  (x,t)\in \Omega^1,\\
\; -\infty \qquad  \text{for } \quad (x,t)\in \left(\R^d\times (0,\infty)\right)\backslash\overline{\Omega^1},\\
\end{cases}
\end{equation}
with $u^1$ and $\Omega^1$ defined by \fer{eq:u1} and  \fer{omeg} respectively. In particular, $u(x,t)\rightarrow u_m$ as $(x,t)\rightarrow \p \Omega^1$.
\end{theorem}

Before we begin with the proof, we present and discuss below several remarks and observations which are important to explain the meaning of the results.

Firstly, by ``uniform convergence'' to $-\infty$, we mean $\limsup_{\e\to 0, y\to x, s\to t }u_\e(y,s)=-\infty$.
Secondly, the $u$ associated with the open set $\Omega^1$ is the maximal solution to \fer{eq:hjs}. Indeed any other solution $\widetilde u$, with the corresponding open set $\widetilde  \Omega$, satisfies $\widetilde u\leq u^1$ and thus $\widetilde \Omega\subset \Omega^1$ and $\widetilde u\leq u$. It also satisfies the Dirichlet and state constraint boundary conditions. To verify the latter we notice, using the standard optimal control formula (\cite{Lionsbook:82, WF.HS:93,MB.ICD:96}), that
$$
u^1(x,t)=\sup_{\underset{x(t)=x}{(x(s),s)\in \R^d\times[0,\infty)}}
\left\{\int_0^t \left( -\f{|\dot x(s)|^2}{4}+R(x(s))\right) ds+u_0(x(0)): x\in C^1([0,t];\R^d) \right\}.
$$

If $\widetilde x(\cdot)$ is an optimal trajectory, the dynamic programming principle implies that, for any $0<\tau<t$,
$$
u^1(x,t)= \int_\tau^t \left(-\f{|\dot {\widetilde x}(s)|^2}{4}+R(\widetilde x(s))\right) ds+u^1(\widetilde x(\tau),\tau).
$$

Since $R$ is nonpositive, $u^1$ is decreasing along the optimal trajectory. It follows that, if $ u^1(x,t)>u_m$,  then, for all $0\leq \tau <t$, $u^1( \widetilde x(\tau), \tau)>u_m$.

Hence, for all $(x,t)\in \Omega^1$,
$$
u(x,t)=\sup_{\underset{x(t)=x}{(x(s),s)\in  \Omega^1}}
\left\{\int_0^t
\left(-\f{|\dot x(s)|^2}{4}+R(x(s)) \right)ds+u_0(x(0)):x\in C^1([0,t];\R^d) \right\},
$$
and, therefore, $u$ verifies the state constraint condition.

Finally, the limit $u$ does not depend on the details of the singular death term. In particular it is the same
if we replace in \fer{eq:sroot} ${n_\e \exp(\e^{-1} u_m)}^{1/2}$ by $n_\e^\gamma \exp(\e^{-1}\gamma u_m)$ with  $0<\gamma<1$. Hence, the value $\gamma=1/2$ is irrelevant.
\\

We continue with the
\begin{proof}[Proof of Theorem \ref{thm:principal}]
As already discussed in the introduction, we know that $u_\e \leq u_\e^1$ but we cannot obtain directly the other inequality in the limit $\e\to0$.  It is therefore necessary to introduce a pair of  auxiliary functions $v_\e^A$ and $v_\e^{A,1}$ which converge, as $\e\to0$, in $\mathrm{C}(\R^d\times(0,\infty))$ to $\max(u^1,-A)$.
Using this information for appropriate values of the parameter $A$, we then prove that, as $\e\to0$,
$u_\e\to u^1$ locally uniformly in the open set
\beq \label{defA}
\caa=\{(x,t):\;u^1(x,t)>u_m\},
\eeq
and
$u_\e\rightarrow -\infty$ locally uniformly in the open set
\beq \label{defB}
\cbb=\{(x,t):\;u^1(x,t)<u_m\}.
\eeq

To this end, for any $A$ such that
\beq\label{A}
0< A<-u_m,
\eeq
we consider the functions $v_\e^A$ and $v_\e^{A,1}$ given by
\beq
\label{defv}
n_\e+\exp(\frac{-A}{\e})=\exp(\frac{v_\e^A}{\e}) \quad \text{ and } \quad n_\e^1+\exp(\frac{-A}{\e})=\exp(\frac{v_\e^{A,1}}{\e}).
\eeq

We have:
\begin{proposition} \label{thm:prop}
Assume (\ref{assum:R}), (\ref{test}), \fer{nonpositive} and \fer{A}.
As $\e\to0$, the families $(v_\e^{A,1})_{\e>0}$ and $(v_\e^A)_{\e>0}$ converge in  $\mathrm{C}(\R^d\times[0,\infty))$ to the unique solution $v^{A,1}=\max(u^1,-A)$ of
\begin{equation}
\label{variation}
\begin{cases}
\min\big(v^{A,1}+A, v_t^{A,1}-|D v^{A,1}|^2-R\big)=0 \ \text{ in } \ \R^d\times(0,\infty),\\[4mm]
v^{A,1}=\max(u^0,-A) \ \text{ on } \ \R^d\times \{0\}.
\end{cases}
\end{equation}
\end{proposition}

We postpone the proof to the end of this section and next
we prove the convergence of the family $(u_\e)_{\e}$ in the sets $\caa$ and $\cbb$. We begin with the former.

Fix $(x_0,t_0)\in \caa$. By the definition of $\caa$ we have $u^1(x_0,t_0)>u_m$ and, hence, we can choose $A$ such that $u^1(x_0,t_0)>-A>u_m$. Proposition \ref{thm:prop} yields that, as $\e\to0$ and uniformly in any neighborhood of $(x_0,t_0)$,
$$
v_\e^{A}\to v^{A,1}=\max(-A,u^1)=u^1.
$$
Using the latter, the choice of $A$ and the fact that
$$
u_\e=v_\e^{A}+\e\ln \big( 1-\exp(\e^{-1}(-A-v_\e^{A})) \big),
$$
we deduce that, as $\e\to0$,  $u_\e \to u^1$ uniformly in any neighborhood of $(x_0,t_0)$.

Next we consider the limiting behavior in the set $\cbb$. To this end, observe that,
using \fer{defue} and \fer{defue1}, we find $u_\e\leq u_\e^1$ and, thus, passing to the limit in the viscosity sense, $\overline u\leq u^1$ and
$$
\overline u< u_m \quad \text{ in } \quad \cbb.
$$

Assume that, for some $(x_0,t_0)\in \cbb$,
$\overline u(x_0,t_0)> -\infty$. Since $\overline u$ is upper semicontinuous (see \cite{GB:94}), there exists a family  $(\phi_\al)_{\al > 0}$ of smooth functions such that $\overline u-\phi_\al$ attains a strict local maximum at some $(x_\al,t_\al)$ and, as $\al\to0$,
$$
(x_\al,t_\al)\to (x_0,t_0), \quad  \overline u(x_\al,t_\al)\geq \overline u (x_0,t_0) \quad \text { and} \quad
\overline u(x_\al,t_\al)\rightarrow \overline u(x_0,t_0).
$$

It follows that there exists points $(x_{\al,\e}, t_{\al,\e})$  such that
$u_\e-\phi_\al$ attains a local maximum at $(x_{\al,\e}, t_{\al,\e})$,
$(x_{\al,\e},t_{\al,\e})\to (x_\al,t_\al)$ as $\e\to0$, and, in view of \fer{eq:ue}, at $(x_{\al,\e}, t_{\al,\e})$,
$$
\phi_{\al,t} -\e \Delta \phi_\al-|D \phi_\al|^2-R\leq -\exp((2\e)^{-1} (u_m-u_\e)).
$$

Letting $\e \to 0$ we find that at $(x_\al,t_\al)$
$$
\phi_{\al,t}-|D \phi_\al |^2-R\leq \limsup_{\e \rightarrow 0}[-\exp ((2\e)^{-1}(u_m-u_\e(x_{\al,\e},t_{\al,\e})))] .
$$

The definition of $\overline u$ yields  $$\limsup_{\e \rightarrow 0}u_\e(x_{\al,\e},t_{\al,\e})\leq \overline u(x_\al,t_\al)$$ and, since, for $\alpha$ sufficiently small, $\overline u(x,t)< u_m$,  we have
$$\overline u(x_\al,t_\al)<u_m \ \ \text{ and } \ \ \limsup_{\e \rightarrow 0}[-\exp((2\e)^{-1}(u_m-u_\e(x_{\al,\e},t_{\al,\e})))]=-\infty$$
and, finally, at $(x_\al,t_\al)$,
$$
\phi_{\al,t}-|D \phi_\al |^2-R\leq -\infty,
$$
which is not possible because $\phi_\al$ is a smooth function.

The claim about the uniform convergence on compact subsets is an immediate consequence of the upper semicontinuity of
$\overline u$ and the previous argument.
\end{proof}

We conclude the section with the proof of Proposition \ref{thm:prop}. Since it is long, before entering in the details, we briefly  describe the main steps. We begin by establishing independent of $\e$ bounds on the family $(v^A_\e)_{\e}$. Then we show that the half-relaxed limits $\overline {v}^\al$ and $\underline {v}^\al$ are respectively sub and supersolutions of (\ref{variation}). We conclude by identifying the limit.
\begin{proof}[Proof of Proposition \ref{thm:prop}]
By the definition of $v_\e^A$, we have $v_\e^A>-A$ and, thus, the family $(v_\e^A)_\e$ is bounded from below.

To prove an upper bound we first notice that, on $\R^d\times\{0\}$,
\beq \label{lnve0}
v_\e^A=u_\e^0+\e\ln(1+e^{\frac{-A-u_\e^0}{\e}})\quad \text{ and } \quad v_\e^A=-A+\e\ln(1+e^{\frac{A+u_\e^0}{\e}}),
\eeq
hence,
$$v_\e^A \leq \max(u_\e^0+\e \ln(2),-A+\e \ln(2)) \quad \text{ on } \quad \R^d \times \{0\},$$
and, finally, in view of \fer{test},
$$ v_\e^A \leq C_A \quad \text{ on } \quad \R^d\times \{0\},$$
for $C_A>0$ such that $\max(-A,u_\e^0)\leq C_A$.

Moreover, since $R \leq 0$,  we have
\begin{equation}\label{eq:2}
v_{\e,t}^A-\e\Delta v_\e^A -|D v_\e^A|^2=\frac{n_\e}{n_\e+\exp(\frac{-A}{\e})}R-\frac{{(\beta_\e n_\e)}^{1/2}}{n_\e+\exp(\frac{-A}{\e})}\leq 0 \quad \text{ in } \quad \R^d\times (0, \infty).
\end{equation}

It follows from the maximum principle that
\begin{equation}\nonumber
v_\e^A \leq C_A+\e \ln(2) \quad \text{ in } \quad \R^d\times (0, \infty).
\end{equation}

Next we show that $\underline{v}^A$ is a supersolution of \fer{variation}. Since $u_m< -A$
and
$$
\frac{{(\beta_\e n_\e)}^{1/2}}{n_\e+\exp(\frac{-A}{\e})}\leq \frac{{(\beta_\e n_\e)}^{1/2}}{2{n_\e\exp(\frac{-A}{\e})}^{1/2}}=\frac{1}{2}\exp(\frac{u_m+A}{2\e}),
$$
as $\e\to0$ and uniformly on $\R^d\times (0, \infty)$, we have
\beq \label{uniformly}
\frac{{(\beta_\e n_\e)}^{1/2}}{n_\e+\exp(\frac{-A}{\e})}\to 0.
\eeq

From \fer{nonpositive}, (\ref{eq:2}) and
$$
0\leq \frac{n_\e}{n_\e+\exp(\frac{-A}{\e})}\leq 1,
$$
we then deduce that, in $\R^d\times (0, \infty)$,
\begin{equation}\label{sursol}
v_{\e,t}^A-\e\Delta v_\e^A -|D v_\e^A|^2\geq R-O(\e),
\end{equation}
while by the definition of $v_\e^A$ we also have
\begin{equation}
\label{vinf}v_\e^A+A\geq 0.
\end{equation}

Combining (\ref{sursol}) and (\ref{vinf}) and using  the basic stability properties of the viscosity solutions (see \cite{GB:94}) we find that the lower semicontinuous function $\underline{v}^A$ is a viscosity supersolution of (\ref{variation}).

To prove that $\overline{v}^A$ is a subsolution to \fer{variation}, following classical arguments from the theory of viscosity solutions (see \cite{GB:94}), we fix a smooth $\phi$ and
assume that $\overline{v}^A-\phi$ has a strict local maximum at $(x_0,t_0)$. It follows  that there exists a family, which for notational simplicity we denote again by $\e$, of points $(x_{\e},t_{\e})_{\e>0}$ in $\R^d\times (0,\infty)$ such that  $v_{\e}^A -\phi$ has a local maximum at $(x_{\e},t_{\e})$, and,  as $\e\to0$,  $(x_{\e},t_{\e})\to (x_{0},t_{0})$ and $v_{\e}^A(x_{\e},t_{\e})\to \overline{v}^A(x_{0},t_{0})$. 

We also know, still using \fer{eq:2} and \fer{uniformly}, that $v_{\e}^A$ solves
$$
v_{\e, t}^A-\e\Delta v_{\e}^A -|D v_{\e}^A|^2=\big(1-\exp(\frac{-A-v_\e^A}{\e})\big)R-O(\e).
$$

It then follows that, at $(x_{\e},t_{\e})$,
\begin{equation}\label{eq:1}
\phi_t-\e\Delta \phi -|D \phi|^2-\big(1-\exp(\e^{-1}(-A-v_\e^A))R\leq O(\e).
\end{equation}

Recall that $\lim_{\e\rightarrow 0} v_\e^A(x_{\e},t_{\e})=\overline v^A(x_0,t_0)\geq -A$. Hence, if
$\overline v^A(x_0,t_0)> -A$, then $$\lim_{\e\rightarrow 0}\exp(\e^{-1}(-A-v_\e^A(x_{\e},t_{\e})))=0.$$

From this and (\ref{eq:1}) we deduce that, if $\overline{v}^A(x_0,t_0)>-A$,  then, at $(x_0,t_0)$,
\begin{equation}\nonumber
\phi_t -|D \phi|^2-R\leq 0.
\end{equation}

Next we show that $\overline v$ and $\underline v$ satisfy the appropriate initial conditions. Indeed,
in view of  \fer{test} and \fer{lnve0}, we know that, as $\e\to0$,
$$
v_\e^{A}\to \max(-A,u^0) \quad \text{ on } \quad \R^d\times\{0\}.
$$

It also follows from a classical argument in theory of viscosity solutions (\cite{GB:94,GB.LE.PS:90}) that, on $\R^d\times\{0\}$,
$$
\overline{v}^A-\max(-A,u^0)\leq 0 \quad \text{ and } \quad \underline{v}^A-\max(-A,u^0)\geq 0.
$$
and, hence, $\overline{v}^A$ and $\underline{v}^A$ satisfy respectively the discontinuous viscosity subsolution and supersolution initial condition corresponding to (\ref{variation}). 

We already know from the definition of $\overline{v}^A$ and $\underline{v}^A$ that
$
\underline{v}^A\leq \overline{v}^A,
$
while
from the comparison property for (\ref{variation}) in the class of semicontinuous viscosity solutions (see \cite{MB.ICD:96,GB:94,C.I.L:92}) we conclude from the steps above that
$
\overline{v}^A\leq \underline{v}^A$ in $\R^d\times(0,\infty).
$
Hence $\underline{v}^A=\overline{v}^A=v^{A,1}$ is the unique continuous viscosity solution of (\ref{variation}) and, consequently, the  families $v_\e^A$ and $v_\e^{A,1}$ converge, as $\e\to0$ and locally uniformly,  to $v^{A,1}$.

Combining \fer{defue} and \fer{defv} we find
$$
v_\e^{A,1}=u_\e^1+\e\ln(1+\exp^{\frac{-A-u_\e^1}{\e}}) \quad \text{ and } \quad v_\e^{A,1}=-A+\e\ln(1+\exp^{\frac{A+u_\e^1}{\e}}).
$$

Moreover, as we already explained it in the introduction (see (\ref{eq:simple})--(\ref{eq:u1})),  we know that, as $\e\to0$,  $u_\e^1\to u^1$ locally uniformly. Hence, always for $A<-u_m$, we obtain that, as $\e\to0$,
$$
v_\e^{A,1} \to \max(u^1,-A) \quad \text{ locally uniformly in } \quad \R^d\times[0,\infty).
$$

It also follows that the family $(v_\e^A)_{\e>0}$  converges, as $\e\to0$, locally uniformly to $v^{A,1}=\max(u^1,-A)$.
\end{proof}

\section{General rate}
\label{Rgeneral}

When $R$ changes sign, the situation is much more complicated and \fer{Rneglim} does not hold in general. In this case we are able to provide only inequalities for the half-relaxed limits of
the family $(u_\e)_{\e>0}$.  These estimates are used later to characterize the limit when $R$ is positive.

Fix $u_0$, $\delta>0$ and recall that  $u^1$ is the solution of (\ref{eq:u1}) with $u^1=u_0$ on $\R^d\times\{0\}$. We introduce next the family $(u_i^{\delta}[u_0], \cac_i^{\delta}[u_0], \Omega_i^{\delta}[u_0])_{i\in\Z^+}$
which is defined iteratively. To this end, for $i=1$, let
\beq\label{step1}
u_1^{\da}[u_0]=u^1, \  \cac_1^{\da}[u_0]=\R^d\times [0,\infty) \ \text{ and } \ \Omega_1^{\da}[u_0]=\{(x,t)\in \R^d\times [0,\infty):  u_1^{\da}[u_0](x,t) > u_m-{\da}\} ,
\eeq
and, given
$u_i^{\da}[u_0], \cac_i^{\da}[u_0]$ and $\Omega_i^{\da}[u_0]$,
$u_{i+1}^{\da}[u_0]:\R^d\times [0,\infty)\rightarrow \R \cup \{-\infty\}$ is defined by
\beq \label{def:-da}
\begin{array}{rl}u_{i+1}^{\da}[u_0](x,t)=
\sup& \left\{ \displaystyle\int_0^t \left[-\f {|\dot x(s)|^2}{4}+R\left(x(s)\right)\right]ds+u_0\left(x(0)\right)\; : \right.\\
&\\
&\left. x\in \mathrm{C}^1([0,t];\R^d),\, (x(s),s)\in \Omega_i^{\da}[u_0]  \;\text{for all }s\in [0,t],\, x(t)=x \right\},
\end{array}
\eeq
with
\beq \label{def:cac-da}
\cac_{i+1}^{\da}[u_0]=\{(x,t)\in \Omega_i^\da[u_0]:  \; u_{i+1}^{\da}[u_0](x,t)>-\infty\}
\eeq
and
\beq\label{def:omega}
\Omega_{i+1}^{\da}[u_0]=\{(x,t)\in \Omega_i^\da[u_0]: \;u_{i+1}^{\da}[u_0](x,t) > u_m-\da\}\subset\cac_{i+1}^\da[u_0].
\eeq

It follows that, in general, $\cac_{i+1}^{\da}[u_0]\subseteq\Omega_i^{\da}[u_0].$
The inclusion may be, however, strict, i.e., they may exist
points $(\bar x,\bar t)\in \Omega_i^\da[u_0]$ which cannot be connected to $\R^d\times\{0\}$ by a $C^1$ trajectory staying, for all $s\in[0,t]$ in $\Omega_i^\da[u^0]$. (See Figure \ref{fig: omega}.)

Moreover  \fer{assum:R}, \fer{step1} and classical considerations from the optimal control theory (\cite{Lionsbook:82, WF.HS:93,MB.ICD:96,PC:09}) yield that,
for all $i\in\Z^+$,
the sets $\cac_i^{\da}[u_0]$ and ${\Omega_i^{\da}}$ are open and $u_i^{\da}[u_0] \in \mathrm{C}(\cac_i^{\da}[u_0]).$

Note that the state constraint boundary condition, i.e., the requirement that the trajectories stay inside the domain, is hidden in the control formula. We do not write it, however, explicitly, because, to the best of our knowledge, there is no general theory, as in \cite{HS:86}, for state constraint problem with time varying and nonsmooth domains. Note that in our context we have no regularity properties for these domains.

\begin{figure}
 \begin{center}
   \input{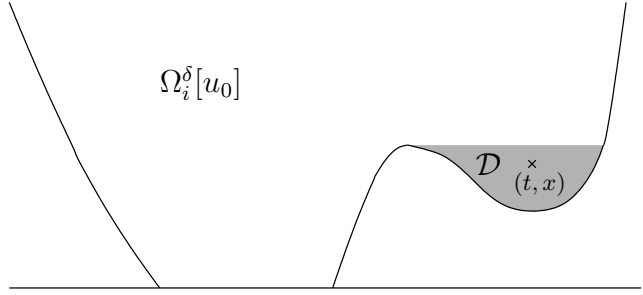}
   \\[-2mm]
   \caption{An example of the space-time set $\Omega_i^\da[u_0]$. The point $(x, t)\in \Omega_i^\da[u^0]$ cannot be connected to $\R^d\times\{0\}$ by a $C^1$ trajectory $(x(s),s)_{s\in [0,t]}$ staying within $\Omega_i^\da[u^0]$. More generally, for the points in the grey area, called $\cad$, there is no admissible trajectory. We have indeed $\cac_{i+1}^{\da}[u_0]=\Omega_i^\da[u_0]\setminus\cad$.}
   \label{fig: omega}
 \end{center}
\end{figure}

Given $\cac_{i+1}^{\da}[u_0]$ as in \fer{def:cac-da}, it turns out that $u_{i+1}^{\da}[u_0]$
is the minimal viscosity solution to
\beq \label{-da}
\begin{cases}
u_{i+1,t}^{\da}[u_0]=|D u_{i+1}^{\da}[u_0]|^2+R \quad  \text{ in } \quad \cac_{i+1}^{\da}[u_0],
\\[2mm]
u_{i+1}^{\da}[u_0]= u_0 \quad \text{ in } \cac_{i+1}^{\da}[u_0]\cap(\R^d\times\{0\}).
\end{cases}
\eeq

Indeed using standard arguments from optimal control theory (see, for example, \cite{MB.ICD:96,GB:94}), we may easily see that $u_{i+1}^\da[u_0]$ satisfies the dynamic programming principle. The latter, as usual, implies that $u_{i+1}^\da[u_0]$ is a viscosity solution of \fer{-da}. The proof of the fact that $u_{i+1}^\da[u_0]$ is a minimal solution to \fer{-da} in Appendix \ref{ap:control}.

The family $(u_i^{\da}[u_0])_{i\in\Z^+,\delta>0}$ is nonincreasing in both $i$ and $\delta$. Therefore there exists $U^{\da}[u_0]\geq -\infty$, which is itself nonincreasing in $\delta$, such that, as $i \to +\infty$,
$
u_i^{\da}[u_0] \searrow~U^{\da}[u_0] \quad \text{ in } \quad \R^d\times[0,\infty).
$

Let $U[u_0]$ be the limit, as $\delta\to0$, of the family $(U^{\da}[u_0])_{\delta>0}$
and, for $\mu>0$, consider the nonincreasing (in $\delta$) family of sets
\beq \label{omega}
\Omega^{\da}[u_0]=\bigcap_{i\in\Z^+} \Omega_i^{\da}[u_0] \quad \text{ and } \quad \Omega[u_0-\mu] =\bigcap_{\da>0}\Omega^\da[u_0-\mu].
\eeq

We have:
\begin{theorem}\label{3.1}
 Let $n_\e$ be the solution to \fer{eq:sroot}, $u_\e=\e \ln(n_\e)$ and assume  \fer{assum:R}. Then,
 for any $\mu>0$,
\beq \label{eq:limsupu}
\overline u\leq  U[u_0] \quad  \text{
in}  \quad  \R^d\times [0,\infty) \quad \text{ and } \quad  U[u_0-\mu]+\mu\leq \underline u \quad \text{
in } \quad \Omega[u_0-\mu].
\eeq
\end{theorem}

Before we present the proof we remark that,
by definition, $u_i^{\da}[u_0]=-\infty$ in $\big({\cac_i^{\da }[u_0]}\big)^c$. Therefore  $U^{\da}[u_0]=-\infty$ in $(\Omega^{\da}[u_0])^c$ and, finally,  $U[u_0]=-\infty$ in $(\Omega[u_0])^c =\big({\bigcap_{i,\da} \cac_i^{\da}[u_0]}\big)^c=\big({\bigcap_\da \Omega^{\da}[u_0]}\big)^c$, and, hence,
$$\overline{u}=-\infty \quad \text{ in } \quad (\Omega[u_0])^c.$$

Moreover, since $u_i^{ \da}[\cdot]\geq u_m- \da$ in $\Omega^{\da}[\cdot]$, by passing to the limit $i\to \infty$ and  $\da\to 0$ we also obtain  $$U[\cdot] \geq u_m \quad \text{ in } \quad  \Omega[\cdot].$$

An important question is whether, as $\mu\to0$,
$U[u_0-\mu] \to U[u_0]$. This is, in general, not true. A counterexample can be found  for $u^0=u_m$ and $R>0$. Then $\Omega_1^{\da}[u_0-\mu]$ cannot touch $\R^d\times\{0\}$ and $u_i^{\da}[u_0-\mu] \equiv -\infty$. Therefore $U[u_0-\mu]\equiv -\infty$ for any $\mu>0$. On the other hand, $u_i^{\da}[u_0] > u_m$ and $U[u_0] =u^1$.

We continue with the

\begin{proof}[Proof of Theorem \ref{3.1}] First we show by induction that, for all $\da>0$ and $i\in\Z^+$,
$
\overline u \;\leq \; u_i^{\da}[u_0].
$

Since $n_\e^1$ is a supersolution to  \fer{eq:sroot}, it follows from the comparison principle that $n_\e\leq n_\e^1$ and, hence, $\overline u \leq u_1^{\da}[u_0]=u^1$.

Next we assume that $\overline u \, \leq \, u_i^{\da}[u_0]$, and, arguing by contradiction, we show, following  an argument similar to that in Section~2, that $\overline u\leq u_{i+1}^\da[u_0]=-\infty$ in $(\Omega_i^{\da}[u_0])^c$.

To this end, suppose that, for some $(x_0,t_0)\in (\Omega_i^{\da}[u_0])^c$, $\overline u(x_0,t_0)> -\infty$.
Since $\overline u$ is upper semicontinuous, there exists a family $(\phi_\al)_{\al>0}$ of smooth functions such that $\overline u-\phi_\al$ attains a strict local maximum at $(x_\al,t_\al)$ and, as $\al\to0$,
$(x_\al,t_\al)\to (x_0,t_0)$, $\overline u(x_\al,t_\al)\geq \overline u (x_0,t_0),$
and, consequently, $\overline u(x_\al,t_\al)\to u(x_0,t_0)$.
It follows that there exist points $(x_{\al,\e}, t_{\al,\e})\in(\R^d\times (0,\infty))$ where $u_\e-\phi_\al$ attains a local maximum 
and, as $\e\to0$,
$(x_{\al,\e},t_{\al,\e})\to (x_\al,t_\al).$

Moreover, in view of  \fer{eq:ue},  at  $(x_{\al,\e},t_{\al,\e})$,
$$
\phi_{\al,t} -\e \Delta \phi_\al -|D \phi_\al|^2-R\leq -\exp((2\e)^{-1}(u_m-u_\e)).
$$

Letting $\e\to0$ yields, at $(x_\al,t_\al)$,
$$
\phi_{\al,t}-|D \phi_\al |^2-R\leq \limsup_{\e \to 0}(-\exp[(2\e)^{-1}(u_m-u_\e(x_{\al,\e},t_{\al,\e}))]).
$$

Since, by the definition of $\overline u$, we have $\limsup_{\e\rightarrow 0} u_\e(x_{\al,\e},t_{\al,\e})\leq \overline u(x_\al,t_\al)$, the induction hypothesis yields that, for $\al$ small enough,
$\overline u(x_\al,t_\al)\leq u_i^{\da}[u_0](x_\al,t_\al) \leq u_m-\da/2$.

It follows that
$$\limsup_{\e \to 0}(-\exp[(2\e)^{-1}(u_m-u_\e(x_{\al,\e},t_{\al,\e}))])=-\infty,$$
and, hence, at $(x_\al,t_\al)$,
$$
\phi_{\al,t} - |D \phi_\al |^2-R\leq -\infty,
$$
which, of course, is not possible because $\phi_\al$ is a smooth function.

Hence we have $\overline u=-\infty$ in $(\Omega_i^{\da})^c$ and, in particular, $\overline u =-\infty$ on $\p \Omega_i^{\da}[u_0]$.
\\

Next we show that $$\overline u\leq u_{i+1}^\da[u_0]=-\infty \quad \text{ in } \quad (\cac_{i+1}^{\da}[u_0])^c.$$

To this end, let $(\bar x,\bar t)\in (\cac_{i+1}^{\da}[u_0])^c\setminus (\Omega_i^{\da}[u_0])^c$. Note that the existence of such a point means that $(\bar x,\bar t)$ cannot be connected to $\R^d\times\{0\}$ by a $\mathrm {C}^1$-trajectory staying in $\Omega_i^\da[u_0]$. Hence $(\bar x,\bar t)$ belongs to a connected component $\cad$ of
$
\omega_i^\da[u_0]=\{(y,s)\in \Omega_i^{\da}[u_0]: \, s\leq \bar t\},
$
which does not touch $\R^d\times\{0\}$. (See Figure \ref{fig: omega}.)

Therefore $\p_p \cad\subset \p \Omega_i^{\da}[u_0]$, where   $\p_p \cad=\{(y,s)\in\p\cad : s<\bar t\}$ is the parabolic boundary of $\cad$. From the previous argument we obtain
\beq \label{pcad}
\overline u=-\infty \qquad\text{on }\p_p \cad.
\eeq

As in \fer{defv}, for $A>0$, we define $w_\e^{A}$ by
$
n_\e+\exp\big( \f{-A}{\e}\big)=\exp\big(\f{w_\e^A}{\e}\big).
$
Arguing as in the previous section, we deduce that, for all $A>0$,
$$
\overline w^A=\max(-A, \overline u) \quad \text{ and } \quad
\min(\overline w^A+A,\, \overline w_t^A-|D \overline w^A|^2-R)\leq 0,
$$
and, in view of  \fer{pcad}, 
$$\bc
\min(\overline w^A+A,\, \overline w_t^A-|D \overline w^A|^2-R)\leq 0 \quad \text{in } \quad \cad,\\
\overline w^A=-A \quad  \text{in } \quad \p_p\cad,
\ec$$
which admits $-A+C_1 t$ as a supersolution for some $C_1>0$.

It follows from the comparison principle that, for all $A>0$,
$$\overline u \leq \overline w^A\leq -A+C_1t \quad \text{in } \quad \cad.$$

Letting $A\to\infty$ yields $\overline u=-\infty$ in $\cad$ and, consequently, $\overline u(\bar x,\bar t)=-\infty$. Observe that $\overline u=-\infty$ in $(\cac_{i+1}^\da[u_0])^c$ implies that $\overline u=-\infty$ on $\p \cac_{i+1}^\da[u_0]\cap (\R^d \times [0,\infty))$.\\

Finally we show that $$\overline u\leq u_{i+1}^{\da}[u_0] \quad \text{ in } \quad \cac_{i+1}^{\da}[u_0].$$

To this end, define $z_\e$ by
$
n_\e+\exp\big( \f{u_{i+1}^\da[u_0]}{\e}\big)=\exp\big(\f{z_\e}{\e}\big)
$
and notice that
$$
z_\e=u_{i+1}^\da[u_0]+\e \ln \left(\exp\left(\f{u_\e-u_{i+1}^\da[u_0]}{\e}\right)+1\right)=
u_\e+\e \ln \left(\exp\left(\f{u_{i+1}^\da[u_0]-u_\e}{\e}\right)+1\right).
$$

It follows that
$$
\overline z=\max(\overline u,\,u_{i+1}^\da[u_0]).
$$

We claim that $\overline z$ is a {subsolution} of
\beq\label{eq:z*} \overline z_t -|D \overline z|^2-R \leq 0 \quad \text{in } \quad \cac_{i+1}^\da[u_0].\eeq

Indeed $u_{i+1}^\da[u_0]$ is a subsolution to \fer{eq:z*} in $\cac_{i+1}^\da[u_0]$ by definition. Moreover if, for some $(\bar x,\bar t)\in \cac_{i+1}^\da[u_0]$, $\overline u(\bar x,\,\bar t)\neq -\infty$, using \fer{eq:ue} and the stability of viscosity subsolutions, we find that $\overline u$ satisfies the viscosity subsolution criteria for \fer{eq:z*} at $(\overline x,\, \overline t)$. Finally, since the maximum of two subsolutions is always a subsolution, we obtain that $\overline z$ is a subsolution of \fer{eq:z*}.

We proceed by noticing that, since, in view of the above,
$$\overline u=-\infty \quad \text{ on } \quad \p \cac_{i+1}^\da\cap (\R^d\times (0,+\infty)),$$
it follows that $$\overline z=u_{i+1}^\da[u_0] \quad \text{ on } \quad \p \cac_{i+1}^\da[u_0]$$
and, hence,
$$\overline z\leq u_{i+1}^\da[u_0] \quad \text{ on } \quad \p \cac_{i+1}^\da[u_0].$$
Therefore, using again the comparison principle for \fer{-da}, we obtain
$$
\overline z\leq u_{i+1}^\da[u_0] \quad \text{in } \quad \cac_{i+1}^\da[u_0]
$$
and we conclude that $\overline u\leq u_{i+1}^\da[u_0].$
\\
%
%
%

Finally, since for all $\delta>0$ and $i\in\Z^+$, we have
$\overline u\leq u_i^\da[u_0]$ it follows that, for all $\da>0$,  $\overline u\leq \lim_{i\rightarrow \infty} u_i^\da[u_0]= U^{\da}[u_0]$. After letting $\da\to0$ we obtain
$$
\overline u  \leq \lim_{\da\rightarrow 0} \,U^{\da}[u_0]=U[u_0] \quad \text{in } \quad \R^d\times(0,\infty),
$$
which concludes the proof of the first part of the claim.

For the second part we need the following lemma which is essentially a result from \cite{EB.RJ:90} that we adapt to our context (see also \cite{GB:BJ}). Its proof is postponed to the end of this section.

\begin{lemma}\label{lem:maxvi}
For all $i\in\Z^+$  the lower semicontinuous function
$v_i^\da=\max(u_i^\da[u_0-\mu]+2\da,\,\underline u)$,
is a supersolution of
\beq \label{supervida}
\begin{cases}
v_{i,t}^\da -|D v_i^\da|^2-R \geq 0 \quad  \text{in } \quad \Omega_i^\da[u_0-\mu],
\\[2mm]
v_i^\da =u_0 \quad  \text{ in  } \{u_0-\mu>u_m-\da\}\cap (\R^d \times \{0\}).
\ec
\eeq
\end{lemma}

Since $u_{i+1}^\da[u_0-\mu]$ is a minimal solution of \fer{-da} in $\cac_{i+1}^\da[u_0-\mu] \subset \Omega_i^\da[u_0-\mu]$
with \\
$u_{i+1}^\da[u_0-\mu]=u_0-\mu$ on $\R^d\times \{0\}$ (see Appendix \ref{ap:control}), it follows that $$ u_{i+1}^\da[u_0-\mu]\leq v_i^\da-\mu \quad \text{  in } \quad  \cac_{i+1}^\da[u_0-\mu],$$ and, hence,
$$u_{i+1}^\da[u_0-\mu]+\mu \leq \max(u_i^\da[u_0-\mu]+2\da,\, \underline u)  \quad \text{in } \quad \cac_{i+1}^\da[u_0-\mu].$$

Letting $i\to\infty$ yields
$$U^\da[u_0-\mu]+\mu  \leq \max(U^\da[u_0-\mu]+2\da,\, \underline u) \quad \text{in } \quad \Omega^\da[u_0-\mu].$$

Choosing $\mu>2\da$ we also get
$$U^\da[u_0-\mu]+2\da  < U^\da[u_0-\mu]+\mu \quad \text{in } \quad \Omega^\da[u_0-\mu],$$
and, therefore,
$$U^\da[u_0-\mu]+\mu\leq \underline u \quad \text{in } \quad \Omega^\da[u_0-\mu].$$

Finally letting $\da\to0$ we obtain
$$
U[u_0-\mu]+\mu=\lim_{\da\rightarrow 0} \;U^\da[u_0-\mu]+\mu\leq \underline u \quad \text{in } \quad \Omega[u_0-\mu].$$

\end{proof}

We conclude with the
\begin{proof}[Proof of Lemma \ref{lem:maxvi}]
The key idea of the proof comes from \cite{EB.RJ:90} and \cite{GB:BJ} and relies on the property that, for concave Hamiltonians, the maximum of two supersolutions is supersolution. Here we reprove this fact in the context of semicontinuous  supersolutions in a space-time domain.

To this end,
fix $i\in\Z^+$ and $(x,t)\in \cac_i^\da[u_0-\mu]$. Since $\cac_i^\da[u_0-\mu]$ is an open set, there exists $\rho>0$ such that $B_\rho(x,t)\in \cac_i^\da[u_0-\mu]$, where $B_\rho(x,t)$ denotes the open ball of radius $\rho$ centered at $(x,t)$.

For $\al>0$, we define
$$
u_i^{\da,\al}(x,t)=\inf_{(y,s)\in B_\rho(x,t) } \{u_i^\da[u_0-\mu](y,s)+ (2\al)^{-1}(|x-y|^2+|t-s|^2)\} \ \ \text{ and } \ \  u_i^{\da,\al,\bt}=u_i^{\da,\al}\ast \chi_\bt,
$$
where $\chi_\bt$ is a standard smoothing mollifier.

Since $u_i^{\da,\al}$ is an inf-convolution of the continuous function $u_i^\da$ (see \cite{GB:94}), it is locally Lipschitz continuous and semi-concave with semi-concavity constant $1/{\al}$.

It follows that $u_i^{\da,\al,\bt}$ is a smooth semi-concave function with semi-concavity constant $1/{\al}$ and
$$\liminf_{\underset{\al,\,\bt\rightarrow 0}{(y,s)\rightarrow(\bar y,\bar s)}}\,u_i^{\da,\al,\bt}(y,s)=u_i^{\da}[u_0-\mu](\bar y,\bar s).$$

Finally, using Jensen's inequality and the concavity of the Hamiltonian, we obtain that, for some $K>0$, $u_i^{\da,\al,\bt}$ is a smooth and, hence, a classical supersolution to
\beq\label{eq:uidaalbt}
u_{i,t}^{\da,\al,\bt}-\e\Delta u_i^{\da,\al,\bt}-|D u_i^{\da,\al,\bt}|^2-R\ast\chi_\bt\geq -K\al-\e/{\al} \quad \text{in } \quad B_\rho(x,t).
\eeq

To prove \fer{supervida} we show that the smooth approximations $v_i^{\da,\al,\bt,\e}$ of $v_i^\da$ in $B_\rho(x,t)$ given by
\beq \label{vda}
n_\e+\exp\big( \f {u_i^{\da,\al,\bt}+2\da }{\e}\big)=\exp\big(\f {v_i^{\da,\al,\bt,\e}}{\e}\big)
\eeq
are almost supersolutions to \fer{supervida} for $\al$, $\bt$ and $\e$ small. Notice that in \fer{vda} we use $2\delta$ instead of $\delta$.

Replacing $n_\e$ by $\exp\big(\f {v_i^{\da, \al,\bt,\e}}{\e}\big)-\exp\big( \f {u_i^{\da,\al,\bt}+2\da }{\e}\big)$ in  \fer{eq:sroot} we get
\begin{align}\nonumber
R n_\e-\bt_\e\sqrt{n_\e}&=\big(v_{i,t}^{\da,\al,\bt,\e}-\e \Delta v_i^{\da,\al,\bt,\e}-|D v_i^{\da,\al,\bt,\e}|^2\big)\exp(\e^{-1}v_i^{\da,\al,\bt})\\
\nonumber&- \big(u_{i,t}^{\da,\al,\bt} - \e \Delta u_i^{\da,\al,\bt}-|D u_i^{\da,\al,\bt}|^2\big)\exp(\e^{-1} (u_i^{\da,\al,\bt}+2\da)),
\end{align}
and, in view of \fer{vda},
\begin{align}\nonumber
v_{i,t}^{\da,\al,\bt,\e}& -\e\Delta v_i^{\da,\al,\bt,\e}-|D v_i^{\da,\al,\bt,\e}|^2
\\ \nonumber &= (u_{i,t}^{\da,\al,\bt} - \e \Delta u_i^{\da,\al,\bt}-|D u_i^{\da,\al,\bt}|^2 -R\ast \chi_\bt)
\exp(\e^{-1}( u_i^{\da,\al,\bt}+2\da-v_i^{\da,\al,\bt,\e}))
\\
\nonumber& +(R\ast \chi_\bt-R)\exp(\e^{-1} (u_i^{\da,\al,\bt}+2\da-v_i^{\da,\al,\bt})) + R - \bt_\e n_\e^{1/2}\exp(-\e^{-1}v_i^{\da,\al,\bt,\e}).
\end{align}

Using that, in view of \fer{vda}, $\exp(\e^{-1}(u_i^{\da,\al,\bt}+2\da-v_i^{\da,\al,\bt,\e}))\leq 1$, and \fer{eq:uidaalbt} we find
\begin{align}\nonumber
v_{i,t}^{\da,\al,\bt,\e}- & \e\Delta v_i^{\da,\al,\bt,\e}-|D v_i^{\da,\al,\bt,\e}|^2 -R\geq  -K\al-\e/\al \\
\nonumber& +(R\ast \chi_\bt-R)\exp(\e^{-1} (u_i^{\da,\al,\bt}+2\da-v_i^{\da,\al,\bt})) - \bt_\e n_\e^{1/2}\exp(\e^{-1}v_i^{\da,\al,\bt}).
\end{align}

Define
$$
v_i^{\da,\al,\bt}(\bar y,\bar s) = \liminf_{\underset{(y,s)\rightarrow (\bar y,\bar s)}{\e\rightarrow 0}}\,v_i^{\da,\al,\bt,\e}(y,s).
$$

Letting  $\e\to0$ and using the stability of viscosity supersolutions we obtain
\begin{align}\label{lime}
v_{i,t}^{\da,\al,\bt}-|D v_i^{\da,\al,\bt}|^2 -R&\geq -K\al \\
\nonumber &+\liminf_{\underset{(y,s)\rightarrow (\bar y,\bar s)}{\e\rightarrow 0}}\,{[ (R\ast \chi_\bt-R)\exp(\e^{-1}(u_i^{\da,\al,\bt}+2\da-v_i^{\da,\al,\bt}))- \bt_\e n_\e^{1/2}\exp(\e^{-1}v_i^{\da,\al,\bt})]}.
\end{align}

Recalling that  $u_i^\da+2\da > u_m + \da $ in $\Omega_i^\da[u_0-\mu]$, we deduce that, as $\e,\al,\bt\to0$,
in $B_\rho(x,t)$,
\beq\label{ineq}
\bt_\e n_\e^{1/2}\exp(-\e^{-1}v_i^{\da,\e})=\bt_\e n_\e^{1/2}(n_\e+\exp(\e^{-1}u_i^{\da,\e}))^{-1} \leq
(1/2)\bt_\e\exp(-(2\e)^{-1}u_i^{\da,\al,\bt})\to 0.
\eeq

Moreover, as $\e,\bt\to0$, we also have
\beq \label{rchi}R\ast \chi_\bt-R \to 0, \ \ \  \exp(\e^{-1}( u_i^{\da,\al,\bt}+2\da-v_i^{\da,\al,\bt}))<1 \ \ \text{ and } v_i^\da(\bar \eta, \bar s)=\liminf_{\underset{(y,s)\rightarrow (\bar y,\bar s)}{\al,\bt\rightarrow 0}}\,v_i^{\da,\al,\bt}(y,s).
\eeq

Using \fer{lime}, \fer{ineq}, \fer{rchi},
and the stability of viscosity supersolutions we find
$$
v_{i,t}^{\da}-|D v_i^{\da}|^2 -R\geq 0 \quad \text{in } \quad {B}_\rho(x,t).
$$

Since all the above hold for
all $(x,t)\in \Omega_i^\da[u_0-\mu]$, it follows that the lower semicontinuous function $v_i^\da$ is a supersolution to
$$
v_{i,t}^\da -|D v_i^\da|^2-R \geq 0 \quad \text{in } \quad\Omega_i^\da[u_0-\mu], \ \
v_i^\da =u_0 \quad  \text{for } \quad \{u_0-\mu>u_m-\da\}\cap (\R^d\times \{0\}).
$$
\end{proof}

\section{Constant rate  }
\label{sec:Rcons}
Here we assume that the rate is a constant, i.e.,
\beq\label{constant}
R(x)=R \quad \text{in } \R^d,
\eeq
and, in addition, setting $O=\{x\in\R^d: \,u_0(x)> u_m\}$,  we have
\beq\label{bdr}
\overline{O} = \{x\in\R^d: \,u_0(x)\geq u_m\}.
\eeq

We have:
\begin{theorem}
\label{th:Rconst} Assume \fer{constant} and \fer{bdr}. Then
\beq\label{limuRcons}\lim_{\e \rightarrow 0}u_\e(x,t)=U[u_0](x,t) \quad \text{locally uniformly in  } \left(\R^d\times [0,\infty)\right)\setminus \left\{(x,t)\,|\,U[u_0](x,t)= u_m\right\},\eeq
with
\beq\label{Orconst}\Omega[u_0]=\{ (x,t)|\,\sup_{y\in \overline{O}}\,\{-\f {|x-y|^2}{4t}+Rt+ u_0(y)\}\geq u_m \},\eeq
and
\beq\label{Urconst}U[u_0]=\bc\sup_{y\in \overline{O}}\, \left\{ -\f{|x-y|^2}{4t}+Rt+u_0(y)\right\} \quad \text{ if }\quad (x,t)\in\Omega[u_0],\\[2mm]
-\infty \quad \text{otherwise.}\ec\eeq
\end{theorem}

We notice that, if $R<0$, then one can obtain \fer{limuRcons} from \fer{Rneglim} and the dynamic programming principle. We also remark that, in particular, Theorem \ref{th:Rconst} shows that the limit of the family $(u_\e)_{\e>0}$ is not, in general, given by \fer{Rneglim}.
 We refer to Appendix \ref{ap.nuniq} for an explicit example.

\begin{proof}[Proof of Theorem \ref{th:Rconst}]
When the rate $R$ is constant, after one iteration of \fer{def:-da}, \fer{def:cac-da} and \fer{def:omega} we find both  the set $\Omega^\da[u_0]$ and the function $U^\da[u_0]$, since for all $i>1$, $j>2$ and $\da>0$, $\Omega_i^\da[u_0]=\cac_j^\da[u_0]=\Omega^\da[u_0]$ and $u_i^\da[u_0]=U^\da[u_0]$.

Indeed, every optimal trajectory in $\cac_2^\da$ is  a straight line connecting a point in $\Omega_2^\da$ to a point in $I^\da=\{x\in\R^d: \, u_0(x)>u_m-\da\}$ and, hence, it  is included in $\Omega_2^\da$. This follows from the observation that    $$\phi(x,t) =-\f{|x-c|^2}{4t}+Rt+u_0(c)$$ is concave in $(x,t)$ and, therefore, all the optimal trajectories of the points in $\Omega_2^\da$ are included in $\Omega_2^\da$. It follows that $\Omega_2^\da=\cac_3^\da$, $u_2^\da=u_3^\da$ and consequently $\Omega_2^\da=\Omega_3^\da$. By iteration we obtain, for all $i>2$, $\Omega_2^\da=\Omega_i^\da=\Omega^\da$ and $u_2^\da=u_i^\da=U^\da$.

Using  \fer{def:omega} and \fer{def:-da} we see that, for all $i\geq2$,
\beq\label{Odarconst}\Omega^\da[u_0]=\Omega_i^\da[u_0]=\{ (x,t): \,\sup_{y\in I^\da}\,
\{-\f {|x-y|^2}{4t}+Rt+ u_0(y)\}>u_m-\da \},\eeq
and
\beq \label{Udarconst} U^\da[u_0]=u_i^\da[u_0](x,t)=\bc \sup_{y\in I^\da}\, \left\{-\f{|x-y|^2}{4t}+Rt+u_0(y)\right\} \quad  \text{if } \quad (x,t)\in\Omega^\da[u_0],\\[2mm]
-\infty \quad  \text{otherwise.}\ec\eeq

It is easy to verify that \fer{-da} holds, since, for all $i>2$ and $\da>0$,
$$
u_{i,t}^{\da}[u_0]-|D u_{i}^{\da}[u_0]|^2-R=0 \quad \text{in }\quad \Omega_i^\da[u_0]=\cac_{i+1}^\da[u_0].
$$

Letting $\da\to0$ in \fer{Odarconst} and \fer{Udarconst} we obtain \fer{Orconst}--\fer{Urconst}. (See Figure \ref{fig: rconst}.)

\begin{figure}
 \begin{center}
  \input{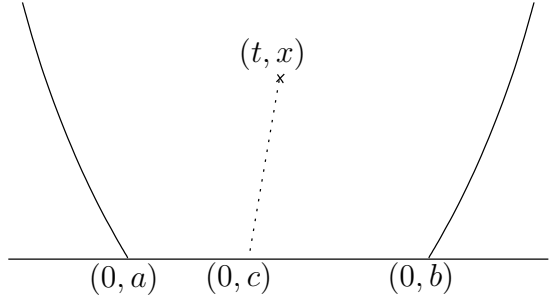}
  \\[-4mm]
   \caption{The case with $R(x)=R$ a positive constant, $\{x\in\R| \,u_0(x)> u_m\}=(a,b)$ and $u_0(\cdot)\geq u_m $ on an interval $[a,b]$. Then $\Omega=\cup_{d\in[a,b]}\{(x,t)\,|\,-\f {|x-d|^2}{4t}+Rt+ u_0(d)\geq u_m\}$, the optimal trajectories are straight lines and   $U(x,t)=-\f{|x-c|^2}
{4t}+Rt+u_0(c)$, where $c$ is a point where the maximum in \fer{Urconst} is attained.}
   \label{fig: rconst}
 \end{center}
\end{figure}

We also have
$$
\Omega[u_0-\mu]=\bigcap_{\da>0}\Omega^\da[u_0-\mu]=\{(x,t):\,\sup_{y\in J^\mu}\,\{-\f {|x-y|^2}{4t}+Rt+ u_0(y)\}\geq u_m +\mu\},$$
with
$$
J^\mu=\{(x,t):\,u_0\geq u_m+\mu\}.
$$

It follows that
\beq\label{O-rconst}
\cup_{\mu> 0}\Omega[u_0-\mu]=\{ (x,t):\,\sup_{y\in O}\,\{-\f {|x-y|^2}{4t}+Rt+ u_0(y)\}>u_m \},
\eeq
and
\beq\label{U-rconst}
\lim_{\mu\rightarrow 0^+}U[u_0-\mu]=\bc\sup_{y\in O}\, \{-\f{|x-y|^2}{4t}+Rt+u_0(y)\} \quad  \text{for }\quad (x,t)\in\cup_{\mu> 0}\Omega[u_0-\mu],
\\[4mm]
-\infty \quad \text{otherwise.}\ec
\eeq

We also notice that
\beq\label{supsup}
\sup_{y\in O}\, \{-\f {|x-y|^2}{4t}+Rt+ u_0(y)\}=\sup_{y\in \bar O}\,\{-\f {|x-y|^2}{4t}+Rt+ u_0(y)\}.
\eeq

Comparing \fer{Orconst}, \fer{Urconst} with \fer{O-rconst}, \fer{U-rconst} and using \fer{supsup} we deduce that
$$
\lim_{\mu\rightarrow 0}U[u_0-\mu](x,t)=U[u_0](x,t) \quad \text{for } \quad U[u_0](x,t)\neq u_m,
$$
and, consequently,  $$\lim_{\e \rightarrow 0}u_\e(x,t)=U[u_0](x,t) \quad \text{locally uniformly in } \left(\R^d\times [0,\infty)\right)\setminus \left\{(x,t)\,|\,U[u_0](x,t)= u_m\right\}.$$

\end{proof}

\section{Strictly positive rate}
\label{sec:rpositive}

In this section we study the limiting behavior of the family $(u_\e)_{\e>0}$ when
\begin{equation}\label{positive}
R \geq a>0 \quad \text{ in } \quad \R^d,
\end{equation}
and   show that, in general, the limit is not given by \fer{Rneglim}.

For this we need to assume that, for sufficiently small $\mu>\da>0$, there exists $\rho_{\da,\mu}>0$  such that  \quad
\begin{align} \label{A1}
\lim_{\mu\rightarrow 0}\lim_{\da\rightarrow 0}\rho_{\da,\mu}=0 \quad \text{ and, if  } \quad
u_0(y)>u_m-\da, \quad \text { then } \quad \sup_{|y-z|\leq \rho_{\da,\mu}} u_0(z)> u_m-\da+\mu.
\end{align}

Notice that it is important that $\rho_{\da,\mu}$ is chosen independently of $y$. If $u_0\in \mathrm{C}^1$, \fer{A1} implies $u_m$ is never a  local maximum of $u$.

We have
\begin{theorem}
\label{u-u+}
Assume \fer{positive} and \fer{A1}. Then
\beq\label{Ucont}
\lim_{\e \rightarrow 0}u_\e=U[u_0] \quad \text{locally uniformly in } \quad  \cup_{\mu>0}\, \Omega[u_0-\mu].
\eeq
\end{theorem}

Recall that, in view of Theorem \ref{3.1}, we already know that $\lim_{\e \rightarrow 0}u_\e=-\infty$ in $\Omega[u_0]^c$. 

\begin{proof}[Proof of Theorem \ref{u-u+}]
For $h>\bar h=\f {\mu}{2a}+\f 1 2\sqrt{\f{\mu^2}{a^2}+\f {\rho_{\da,\mu}^2}{a}}$, $(x,t)\in \R^d \times [0,\infty),\; i\geq 1 $ and $\mu,\,\da>0,$ we have
\beq
\label{delay}
u_i^\da[u_0](x,t)\leq u_i^\da[u_0-\mu](x,t+h).
\eeq

We postpone the proof of this inequality to Appendix \ref{app: dalay} and we continue with the ongoing one.

Letting $i\to+\infty$ and $\da, \mu \to 0$ we find, for all $h>0$ and $t>0$,
\beq\label{delay1}
U[u_0](\cdot,\cdot)\leq \lim_{\mu\rightarrow 0^+}U[u_0-\mu](\cdot,\cdot + h).
\eeq

Hence, for all $(x,t)\in \cup_{\mu>0}\Omega[u_0-\mu]$,
$$
U[u_0](x,t)\leq \lim_{\mu\rightarrow 0^+}U[u_0-\mu](x,t+h)\leq \underline u(x,t+h)\leq \overline u(x,t+h) , 
$$
$$
U[u_0](x,t)\leq \liminf_{h\rightarrow 0^+}\underline u(x,t+h)\leq \limsup_{h\rightarrow 0^+}\,\overline u(x,t+h).
$$

The definitions of $\underline u$ and $\overline u$ also imply that
$$
\liminf_{h\rightarrow 0^+}\underline u(x,t+h)=\underline u(x,t) \quad \text{ and } \quad \limsup_{h\rightarrow 0^+}\,\overline u(x,t+h)=\overline u(x,t).
$$

Combining all the above we obtain
$$
U[u_0]\leq \underline u \leq \overline u \quad \text{in } \quad \cup_{\mu>0}\Omega[u_0-\mu].
$$

This last inequality and \fer{eq:limsupu} yield
$
\underline u =\overline u=U[u_0] \quad \text{in } \quad \cup_{\mu>0}\Omega[u_0-\mu],
$
and, hence,
$$
\lim_{\e\rightarrow 0}u_\e =U[u_0] \quad \text{in } \quad \cup_{\mu>0}\Omega[u_0-\mu].
$$
\end{proof}

\section{Conclusions}
\label{sec:conclusion}

We showed that the local uniform limit, as $\e \to 0$, for the parabolic problem \fer{eq:sroot} with finite time extinction is naturally analyzed using the Hopf-Cole change of variables \fer{defue}. The formal limit is the variant \fer{eq:hjs} of the standard eikonal equation.  The new feature is the resulting  quasi-variational inequality with an  obstacle that depends on the solution itself.

The quasi-variational  inequality admits many solutions (see Appendix \ref{ap.nuniq}) and the difficulty is to select the correct additional information. This is easy when the rate  $R$ is  negative, as shown in Section \ref{sec:Rneg}. Indeed, in this case it is enough to enforce the Dirichlet boundary condition on the boundary of the unknown open set $\Omega$ where the liminf of the family $(u_\e)_{\e>0}$ is finite. This is due to the fact that,  for concave Hamiltonians,  the supremum of two supersolutions is still a supersolution.

When the rate $R$ is  positive we do not have easy supersolutions at hand, and the answer is more elaborate. It requires an induction argument which allows us to identify again the limit of the family $(u_\e)_{\e>0}$. The  key ingredient
is a free boundary problem defined through the level set of the solution. The boundary condition for the resulting equation involves state constraints which leads us to study the problem using the related control problem.

If the growth/death rate $R$ changes sign, we can only bound from above and below the half-relaxed limits of the family $(u_\e)_{\e>0}$ by sub and supersolutions $\bar u$  and $\underline u$ respectively. 

In terms of the biological motivation, our results qualitatively mean that the specific form of the survival threshold (a square root here) is irrelevant for the asymptotic problem.
It also shows that  the solution is deeply influenced by the survival threshold except when $R$ is nonpositive. This confirms earlier numerical simulations in \cite{MG.BP:09}.

We conjecture that these upper and lower solutions are in fact equal and the correct setting (implying uniqueness) is to find a pair $(u, \Omega)$ for which we can impose both Dirichlet and state constraints boundary conditions. Both establishing directly these boundary conditions for the half-limits of the family $(u_\e)_{\e>0}$  as well as developing a theory of state constraints boundary conditions for time varying, non-smooth domains are challenging mathematical issues.

\appendix

\section{Non-uniqueness}
\label{ap.nuniq}

To explain the difficulty associated with  \fer{eq:hjs}, we present here  counter-examples for uniqueness and elaborate further conditions. Recall that  the problem is to find pairs $(u,\Omega)$ such that $u$ is a viscosity solutions  to \fer{eq:hjs}.

A first source for non-uniqueness is the value of $u$ on $\p \Omega$. Indeed assume that $R$ and $u_0$ are such that there exists a unique  viscosity solution $u^1$ of \eqref{eq:u1} or, more generally, with $u^1$ defined in \fer{eq:simple} and \fer{defue1}. For all $\eta\geq u_m$, we introduce the pair $(w_\eta, \Omega_\eta)$ given by
$$
\Omega_\eta=\{(x,t):\; u^1(x,t)\geq\eta\} \quad \text{ and} \quad
w_\eta(x,t)=\bc
u^1(x,t) \quad \text{ if } \quad (x,t)\in \Omega_\eta,\\[2mm]
-\infty \quad \text{otherwise.}
\ec
$$

It can be easily verified that $(w_\eta, \Omega_\eta)$ is a viscosity solution of \eqref{eq:hjs}. In order to avoid this artefact, one can add the Dirchlet boundary condition \fer{limbor} which appeared throughout our constructions. However in the next example we see that this Dirichlet condition is not enough to obtain uniqueness. In fact a state constraint boundary condition is hidden behind the property $u^1= -\infty$ in the complement of $\overline\Omega_\eta$ and we do not take it into account here.
\\

Let
$$
R(x)=1 \quad \text{ and } \quad u_0(x)=-x^2.
$$

A simple computation shows that the solution $u^1$ to \eqref{eq:u1} is given by
$$
u^1(x,t)=t-\f{x^2}{1+4t}.
$$

Therefore the first truncation of $u^1$, given by 
$$
\widetilde u(x,t)=\bc
t-\f{x^2}{1+4t} \quad \text{ for } \quad t-\f{x^2}{1+4t}\geq u_m,\\[4mm]
-\infty \quad  \text{otherwise},
\ec
$$
with
$$
\widetilde\Omega=\{(x,t):\widetilde u(x,t)>-\infty\},
$$
is a viscosity solution of \fer{eq:hjs}.
As a matter of fact this is the maximal subsolution to \fer{eq:hjs}, \fer{limbor} but it does not satisfy the state constraint boundary condition. To see this choose $u_m=-0.04$. The point $(1,2)$ is included in $\widetilde\Omega$ since $\widetilde u(1,2)=0.2>{-0.04}$. The optimal trajectory associated to this point, giving the value $\widetilde u(1,2)=0.2$, is the straight line connecting $(0,0.4)$ to $(1,2)$. But  $u_0(0.4)=-0.16< -0.04$. So the point $(0,0.4)$ is not included in $\widetilde\Omega$. Therefore a part of the optimal trajectory of the point $(1,2)$ is not included in $\widetilde\Omega$. Hence $\widetilde u$ does not satisfy the state constraint condition.\\

Following the arguments in Section \ref{sec:Rcons} we  can find a viscosity solution  to \fer{eq:hjs} and \fer{limbor}.
Indeed using \fer{Urconst} it is possible to  compute explicitly the function $U[u_0]=\lim_{\da\rightarrow 0}U^\da[u_0]=\lim_{\da\rightarrow 0}u_2^\da[u_0]$ to find
$$
\breve u(x,t)=\bc
t-\f{x^2}{1+4t} \quad \text{if } \quad -\f {x^2}{(1+4t)^2}\geq u_m,\, t-\f{x^2}{1+4t}\geq u_m,\\[2mm]
t-\f{(x-\sqrt{-u_m})^2}{4t}+u_m \quad \text{ if  } \quad x>0,\,-\f {x^2}{(1+4t)^2}\leq u_m,\, t \geq \f{(x-\sqrt{-u_m})^2}{4t} ,\\[2mm]
t-\f{(x+\sqrt{-u_m})^2}{4t}+u_m \quad \text{ if } \quad x<0,\,-\f {x^2}{(1+4t)^2}\leq u_m,\, t \geq \f{(x+\sqrt{-u_m})^2}{4t},\\[2mm]
-\infty \quad  \text{otherwise,}
\ec
$$
with
$$
\breve\Omega=\{(x,t): \breve u(x,t)>-\infty\}.
$$

From Theorem \ref{th:Rconst} we know that $\breve u$ is indeed the pointwise limit of the family $(u_\e)_{\e>0}$  outside the exceptional set $\{(x,t): \breve u(x,t)=u_m\}$.
\\

However, in general $\widetilde u\neq \breve u$.  Consider, for instance, the value $u_m=-0.04$. Then
$$
\widetilde u(2,1)=0.2,\quad \breve u(2,1)=0.15,
\quad  \widetilde u(2.21,1)=0.02,\quad \breve u(2.21,1)=-\infty.
$$
and, consequently,
$
\breve \Omega \varsubsetneq \widetilde \Omega.
$

On the other hand, according to Section \ref{sec:Rcons},  the state constraint boundary condition is satisfied for $\breve u$, which motivates our conjecture in Section \ref{sec:conclusion}.

\section{$u_i^{\da}[u_0]$ is a minimal solution of \fer{-da} in $\cac_{i}^{\da}[u_0]$}
\label{ap:control}

Here we prove that $u_i^{\da}[u_0]$ is a minimal solution of \fer{-da} in $\cac_{i}^{\da}[u_0]$ by
considering a  supersolution $w\in \cac_{i}^{\da}[u_0]$  of \fer{-da} and showing that

\beq
\label{minimal}u_i^{\da}[u_0]\leq w  \quad \text{in }
\quad \cac_{i}^{\da}[u_0].
\eeq

To this end, we fix $(x,t)\in \cac_i^{\da}[u_0]$ and
assume that $\left(\gamma(\cdot),\cdot\right):[0,t]\rightarrow \Omega_{i-1}^{\da}[u_0]$
is a $\mathrm{C}^1$-trajectory with $\left(\gamma(t),t\right)= (x,t)$.
Since $\cac_i^{\da}[u_0]$ is the set of points that can be connected by a $\mathrm{C}^1$-trajectory in $\Omega_{i-1}^{\da}[u_0]$ to some point in $\R^d\times\{0\}$, it follows that $\gamma$ is included in $\cac_i^{\da}[u_0]$.

For the supersolution $w$, we define, for $s\in [0,t]$, the (clearly) lower semicontinuous function
$\vp(s) = w(\gamma(s),s)$ and we observe 
that $\vp$ is a viscosity supersolution of
\beq
\label{vpsup}\vp'\geq -\f{|\dot{\gamma}|^2}{4}+R(\gamma) \qquad \text{in } \quad (0,t).
\eeq

We postpone the proof of this claim to the end of the present paragraph and we proceed noticing that the function
$$\psi(t)=\int_0^t \big(-\f{|\dot{\gamma}(s)|^2}{4}+R(\gamma(s))\big) ds+w(\gamma(0),0),$$
is a subsolution of \fer{vpsup}. Then using the standard comparison principle of viscosity solutions we obtain
$$w(x,t)=\vp(t)\geq  \int_0^{t} \big(-\f{|\dot{\gamma}(s)|^2}{4}+R(\gamma(s))\big)ds+u_0(\gamma(0)),$$
and, since this is true for any $\mathrm{C}^1$-trajectory $\gamma$ and any $(x,t)\in \cac_i^{\da}[u_0]$,  \fer{minimal} follows.

It remains to prove \fer{vpsup}. Let $\phi\in \mathrm{C}^1((0,t))$ be a test function, assume that $\bar t$ is a strict minimum point of $\vp-\phi$ and  consider the function
$$F_\mu(y,t)=w(y,t)-\phi(t)+\f{|y-\gamma(t)|^2}{\mu^2}+(t-\bar t)^2,$$
which attains a local minimum at a point $(y_\mu,t_\mu)$ such that, as $\mu\to0$,
\beq \label{limmu}
t_\mu-\bar t\rightarrow 0  \quad \text{ and } \quad \f{|y_\mu-\gamma(t_\mu)|^2}{\mu^2}\rightarrow 0.
\eeq

Since $w$  is a supersolution to  \fer{-da}, we have
$$\phi'(t_\mu)+\f {2\big(\gamma(t_\mu)-y_\mu\big)}{\mu^2}\cdot \dot{\gamma}(t_\mu)+2(t_\mu-\bar t)\geq \big|\f{2(y_\mu-\gamma(t_\mu))}{\mu^2}\big|^2+R(y_\mu).
$$

It is immediate that
$$
\phi'(t_\mu)+2(t_\mu-\bar t)\geq -\f{|\dot{\gamma}(t_\mu)|^2}{4}+R(y_\mu),
$$
and, after letting $\mu\to0$, we conclude using \fer{limmu}.


\section{The proof of \fer{delay}}
\label{app: dalay}

We prove by induction on $i$ that, for all $h>\bar h=\f {\mu}{2a}+\f 1 2\sqrt{\f{\mu^2}{a^2}+\f {\rho_{\da,\mu}^2}{a}}$, $i>1, \, \da>0,$ and  $(x,t)\in \R^d\times[0,\infty)$,
$$u_i^\da[u_0](x,t)\leq u_i^\da[u_0-\mu](x,t+h).$$

Recall that $u_1^\da[u_0]=u^1[u_0]$ and $u_1^\da[u_0-\mu]=u^1[u_0-\mu]=u^1[u_0]-\mu$, where $u^1[u_0]$ is the solution of {\fer{eq:u1}}.
Moreover \fer{positive} yields
$$ u^1[u_0](\cdot,t)+ah-\mu\leq u^1[u_0](\cdot,t+h)-\mu=u^1[u_0-\mu](\cdot,t+h).$$

Therefore, for all $h>\bar h\geq \mu/ a$, we have
$$u^1[u_0](\cdot,t)\leq u^1[u_0-\mu](\cdot,t+h),$$
and, consequently,
$$u_1^\da[u_0](\cdot,t)\leq u_1^\da[u_0-\mu](\cdot,t+h).$$

If, for all $h>\bar h$ and $t>0$,
$$u_i^\da[u_0](\cdot,t)\leq u_i^\da[u_0-\mu](\cdot,t+h),$$
it follows that, for all $h>\bar h$,
\beq\label{omegahet}\Omega_i^\da[u_0]+he_t\subset \Omega_i^\da[u_0-\mu],\eeq
where $e_t$ is the unit vector in the direction of time axis.

Fix $(x,t)\in \cac_{i+1}^\da[u_0]\subset\Omega_i^\da[u_0]$ and let $\gamma$ be  a $\mathrm{C}^1$-trajectory in $\Omega_i^\da[u_0]$ connecting $(x,t)$ to a point $(y,0)$ with $u_0(y)>u_m -\da$. It follows from \fer{A1}  that there exists $z\in\R^d$ such that $|z-y|<\rho_{\da,\mu}$ and $u_0(z)>u_m-\da+\mu$. Without loss of generality we can take $u_0(z)\geq u_0(y)$.

The claim is that the trajectory $\widetilde \gamma:[0,t+h]\rightarrow \R^d$ defined by
\beq\label{def:gamma'}\widetilde\gamma(s)=\bc
h^{-1}s(y-z)+z \quad \text{ if } \quad 0\leq s\leq h,\\[2mm]
\gamma(s-h)\quad \text{for } \quad h< s\leq t+h,
\ec\eeq
is included in $\Omega_i^\da[u_0-\mu]$. Indeed  notice that the choice of $\bar h$ yields, for all $h>\bar h$,
$$-\f{|y-z|^2}{4h}+ah\geq \mu\geq0.$$

Consequently, it follows from  \fer{positive} and  the choice of $z$ that the straight line connecting $(y,h)$ to $(z,0)$ is included in $\Omega^\da[u_0-\mu]=\cap_j\Omega_j^\da[u_0-\mu]$, and, in particular, in $\Omega_i^\da[u_0-\mu]$. Therefore, for all $0\leq s\leq h$, the point $(\widetilde\gamma(s),s)$ is included in $\Omega_i^\da[u_0-\mu]$.

Moreover using \fer{omegahet} we find that $\left(\gamma(s),s+h\right)\in\Omega_i^\da[u_0-\mu]$ for all $s\geq  0$. Hence,  for all $h<s$, $(\widetilde\gamma(s),s)\in\Omega_i^\da[u_0-\mu]$, and we conclude that $\widetilde\gamma$ is included in $\Omega_i^\da[u_0-\mu]$.
\\

Next write
\begin{align}\label{gamma'}
\int_0^{t+h}(-\f{|\dot{\widetilde\gamma}(s)|^2}{4}+R(\widetilde\gamma(s)))ds+u_0(z)-\mu&=\int_0^{t}(-\f{|\dot{\gamma}(s)|^2}{4}+R(\gamma(s)))ds\\
\nonumber&+\int_0^{h}(-\f{|\dot{\widetilde\gamma}(s)|^2}{4}+R(\widetilde\gamma(s)))ds+u_0(z)-\mu.
\end{align}

It follows  that
\beq
\label{diff}
\int_0^{h}(-\f{|\dot{\widetilde\gamma}(s)|^2}{4}+R(\widetilde\gamma(s)))ds+u_0(z)-\mu\geq u_0(y).
\eeq

If this is true, then using \fer{def:-da}, \fer{gamma'} and \fer{diff} we obtain,  for all $h>\bar h$ and
$t>0$,
$$
u_{i+1}^\da[u_0](\cdot,t)\leq u_{i+1}^\da[u_0-\mu](\cdot,t+h),
$$
and we deduce \fer{delay}.
\\

It remains to prove \fer{diff}. Since $R\geq a$, in view of \fer{def:gamma'}, we have
$$
u_0(z)\geq u_0(y),  \ \  \int_0^{h}(-\f{|\dot{\widetilde\gamma}(s)|^2}{4}+R(\widetilde\gamma(s)))ds+u_0(z)-\mu \geq -\f{|y-z|^2}{4h}+ah+u_0(z)-\mu,
$$
and, for all $h>\bar h,$
$$
-\f{|y-z|^2}{4h}+ah\geq \mu.
$$

%
%

\end{document}